\documentclass[12pt]{article}
\usepackage{amsmath}
\usepackage{graphicx}
\usepackage{natbib}
\usepackage{url} 
\usepackage{xcolor}
\usepackage[colorlinks,linkcolor=blue,citecolor=blue,urlcolor=blue,filecolor=blue,backref=page]{hyperref}
\usepackage{mathrsfs}
\usepackage{subfigure}
\usepackage{multirow}
\usepackage{amsfonts}
\usepackage{multirow}
\usepackage{authblk}
\usepackage{amsmath}
\usepackage{mathrsfs} 
\usepackage{amsfonts}
\usepackage{amsthm}
\usepackage{bbm}
\usepackage{mathtools} 
\usepackage{xcolor}
\usepackage{cleveref}

\allowdisplaybreaks
\usepackage{graphicx}
\usepackage{float}
\usepackage{amsfonts}
\usepackage{mathrsfs}
\usepackage{fancyhdr}
\usepackage{amsthm}
\usepackage{cases}
\usepackage{amsmath,amssymb,bm}
\usepackage{tikz}
\usepackage{tkz-euclide}

\usepackage{enumitem}
\usepackage{xcolor}
\usepackage{upgreek}
\usepackage{empheq}

\def\bb0{\mathbb{0}}

\def\EE{\mathbb{E}}

\def\RR{\mathbb{R}}

\def\si{{\sigma}}

\def\si{{\sigma}}

\def\beq{\begin{eqnarray}}
\def\eeq{\end{eqnarray}}
\def\bsp{\begin{equation}
\begin{split}}
\def\esp{\end{split}
\end{equation}}
\def\ba{\begin{align}}
\def\ea{\end{align}}

\newcommand{\lc}{\left(}
\newcommand{\rc}{\right)}
\newcommand{\lk}{\left[}
\newcommand{\rk}{\right]}

\newcommand{\ltd}{\left\{ }
\newcommand{\rtd}{\right\} }

\theoremstyle{Def}

\numberwithin{Thm}{section}

\numberwithin{equation}{section}

\newtheorem{Theorem}{Theorem}
\newtheorem{theorem}[Theorem]{Theorem}
\newtheorem{definition}[Theorem]{Definition}
\newtheorem{lemma}[Theorem]{Lemma}

\newtheorem{corollary}[Theorem]{Corollary}

\newtheorem{proposition}[Theorem]{Proposition}
\newtheorem{example}[Theorem]{Example}

\newtheorem{remark}[Theorem]{Remark}

\newcommand {\bX} {{\bf{X}}}
\newcommand {\bZ} {{\bf{Z}}}

\newcommand {\bt} {{\bf{t}}}

\newcommand {\bR} {{\bf{R}}}

\newcommand {\E} {\mathbb{E}}
\newcommand {\Var} {\text{Var}}

\newcommand {\DP} {\text{DP}}
\newcommand {\Cov} {\text{Cov}}

\begin{document}

\def\spacingset#1{\renewcommand{\baselinestretch}%
{#1}\small\normalsize} \spacingset{1}


%

\title{\bf Large parameter asymptotic analysis for homogeneous normalized random measures with independent increments}
\author[1]{Junxi Zhang \thanks{junxi3@ualberta.ca}}
\author[2]{Shui Feng \thanks{shuifeng@mcmaster.ca}}
\author[3]{Yaozhong Hu \thanks{yaozhong@ualberta.ca}}

\affil[1, 3]{Department of Mathematical and Statistical Sciences,  University of Alberta, Edmonton, Alberta, Canada, T6G 2G1}
\affil[2]{Department of Mathematics and Statistics, McMaster University, Hamilton, Ontario, Canada, L8S 4L8}

\maketitle

\bigskip
\begin{abstract}
Homogeneous normalized random measures with independent increments (hNRMIs) represent a broad class of Bayesian nonparametric priors and thus  are widely used.
In this paper, we obtain the strong law of large numbers, the central limit theorem and the functional central limit theorem of hNRMIs when the concentration parameter $a$ approaches infinity. To quantify the convergence rate of the obtained central limit theorem, we further study  the Berry-Esseen bound, which turns out to be of the form $O \lc \frac{1}{\sqrt{a}}\rc$. As an application of the central limit theorem, we  present  the functional delta method, which can be employed  to obtain the limit of the  quantile process of hNRMIs. As an illustration of the central limit theorems, we  demonstrate  the  convergence numerically for the Dirichlet processes and the normalized inverse Gaussian processes with various choices of the concentration parameters.
\end{abstract}

\noindent%
 
\spacingset{1} 
\section{Introduction}\label{sec:intro}
Normalized random measures with independent increments (NRMIs)  \citep{Regazzini_2003} represent a large class of Bayesian nonparametric priors and thus have undergone extensive study in Bayesian nonparametrics  community. In particular, many widely used Bayesian nonparametric priors belong to the homogeneous NRMIs class. 
 The notable members 
 include   the  Dirichlet process \citep{ferguson1973}, the $\sigma$-stable NRMIs \citep{kingman1975}, the normalized inverse Gaussian process \citep{lijoi2005b}, the normalized generalized gamma process \citep{lijoi2003normalized, lijoi_2007},  the generalized Dirichlet process \citep{lijoi2005a}, and 
 many more other useful Bayesian nonparametric priors.
 We refer to \citet{hjort2003topics,  lijoi2010models, ghosal2017fundamentals, huzhangreview} for a review of these processes. Many properties of NRMIs have been discovered, for example, the moments are calculated in \citet{james2006conjugacy}, the posterior analysis is presented in \citet{james2009}, the stick-breaking representation of hNRMIs is given in \citet{favaro2016}, the large sample asymptotic analysis is provided in \citet{junxi2023large}.

Asymptotic results, such as  the strong law of large numbers, the central limit theorem, the Berry-Esseen theorem, 
are always  central topics in statistics and in probability theory. Without exception, the  
  asymptotic behaviors  of Bayesian nonparametric priors play crucial roles in Bayesian nonparametric analysis, for example in the construction of credible sets and in functional estimations.  In the Bayesian nonparametric framework, researchers consider asymptotic theorems in two cases: large sample asymptotic analysis and large parameter asymptotic analysis. For the former case, many important contributions have been made, for example, posterior consistency analysis for Bayesian nonparametric priors \citep{ghosal1999consistency, james2008, jang2010, Blasi_2013, junxi2023large}, posterior consistency for the Dirichlet process mixture model in density estimation \citep{ghosal1999posterior, lijoi2005consistency}, Bernstein-von Mises theorems \citep{lo1983, lo1986remark, kim2004, ray2021bernstein, huzhang2022functional,james2008, franssen2022bernstein, junxi2023large} and more. For the latter case, asymptotic theorems are obtained when model parameter is large (or small). In this work, we aim to give asymptotic analysis of hNRMIs when their concentration parameter $a$ approaches infinity. It is worth pointing out that the concentration parameter (or total mass) $a$ for Bayesian nonparametric priors is always very important:  it represents how much confidence one places in the prior information.  For example, if  $\{X_i\}_{i=1}^n$ is  a sample from the Dirichlet process
$\DP(a, H)$, then it is known that  the posterior mean   is $\mathbb{E}\left[P(\cdot)|X_1,\cdots,X_n\right]=\frac{a}{a+n}H(\cdot)+\frac{n}{a+n}\frac{\sum_{i=1}^nX_i}{n}$ and
 the predictive distribution is $\mathbb{P}\left(X_{n+1}\in \cdot|X_1, \cdots, X_n \right)=\frac{a}{a+n}H(\cdot)+\frac{n}{a+n}\frac{\sum_{i=1}^n\delta_{X_i}(\cdot)}{n}$,   
which means that  $a$ plays the key role of  the weight of the prior information. As for the interpretation of the concentration parameter in analysis, it   is of the same order as the inverse of the variance of the  priors (see e.g., Proposition \ref{prop: order of var} in Section \ref{sec.main results} and Remark 3.6 in \citet{huzhang2022functional}). In applications, the concentration parameter   has  diverse  specific meanings, for example, it represents the mutation rate of species in the context of population genetics when the Poisson-Dirichlet distribution is used \citep{dawson2006, feng2010poisson}. 
 Let us mention the following inspiring works pioneered our current research. When the concentration parameter $a \rightarrow \infty$, the strong law of large numbers, the central limit theorem and functional central limit theorem  are obtained for the normalized inverse Gaussian process by \citet{labadi2013}, for the Dirichlet process by \citet{labadi2016}, for stick-breaking priors by \citet{huzhang2022functional}. Weak convergence of the Dirichlet process when the corresponding parameter measure ($\alpha=aH$ in this paper) goes to another measure (could be a zero measure) is discussed in 
 \citet{sethuraman1982} (see e.g., Chapter 4 in \citet{ghosal2017fundamentals} for more details). \citet{dawson2006} present
 the large deviation principle
 for the Poisson-Dirichlet distribution (see e.g., \citep{feng2010poisson} for more details) and give the explicit rate functions when the concentration parameter $a \rightarrow \infty$.
 
 In this paper, we obtain the strong law of large numbers, the Glivenko-Cantelli theorem, the central limit theorem, the functional central limit theorem and the Berry-Esseen bound for the convergence of hNRMIs when the concentration parameter $a \rightarrow \infty$.  Let us mention that  the central limit theorem and functional central limit theorem have been studied  for the Dirichlet process by \citet{labadi2016}, for the normalized inverse Gaussian process by \citet{labadi2013}. Their method   relies  on the marginal distributions of these processes, 
however, most Bayesian nonparametric priors have no explicit marginal distributions. 
 \citet{huzhang2022functional} prove   the strong law of large numbers, the central limit theorems and functional central limit theorems  for a general class, stick-breaking processes,     by using the method of moments     through their stick-breaking representations.  
 This method of moments  is hard to apply to more general    hNRMIs since their moments are complicated  to calculate and use.
Thus, this work  
 makes two new contributions compared to  existing research:  we establish the law of large numbers, the central limit theorem, the functional central limit theorem  for quite general   hNRMIs; and we provide  the Berry-Esseen bound 
 for general hNRMIs, which  has not been discussed in any   previous work.

The organization  of this  paper is as follows. In  Section \ref{sec.NRMIs}, we briefly introduce  the    hNRMIs and   a decomposition of the completely random measures when they are used in the construction of hNRMIs. It is worth pointing out that the proposed decomposition  provides a convenient way to show the asymptotic theorems of hNRMIs.  In Section \ref{sec.main results}, we obtain the strong law of large numbers, the Glivenko-Cantelli theorem, the central limit theorem, the functional central limit theorem and the Berry-Esseen bound for hNRMIs. One direct application of the functional central limit is the Delta method, which can be  used  to find the limit  of the quantile process of hNRMIs on $\mathbb{R}$.  In Section \ref{sec. numerical}, we   numerically  illustrate the    convergence in central limit theorem   along with  $a\to \infty$
for the Dirichlet processes  and the normalized inverse Gaussian processes .  Finally,  Section \ref{sec.discussion} concludes  this paper  with  some problems   that can be studied in the future. For ease of reading, we defer all proofs to  the supplementary material.

\section{Normalized random measures with independent increments}\label{sec.NRMIs}
\subsection{Construction of NRMIs}

Completely random measure (CRM) is introduced by \citet{kingman1967completely} (see  \citep{kingman1993poisson} for more details). It plays an important role in the construction of NRMIs. 

Let $(\Omega, \mathcal{F}, \mathbb{P})$ be any probability space. Let $\mathbb{X}$ be a complete, separable metric space whose Borel $\sigma$-algebra is denoted by $\mathcal{X}$ and let $(\mathbb{M}_{\mathbb{X}}, \mathcal{M}_{\mathbb{X}})$ be the space of all probability measures on $\mathbb{X}$. We further denote $\mathbb{B}_{\mathbb{X}}$ as the space of  boundedly finite  measures on $(\mathbb{X},\mathcal{X})$, namely, for any $\mu \in \mathbb{B}_{\mathbb{X}}$ and any bounded set $A \in \mathcal{X}$ one has $\mu(A) <\infty$. Let $\mathbb{B}_{\mathbb{X}}$ be endowed with the vague or the weak hash topology (convergence of integrals of bounded continuous functions vanishing at infinity) with the Borel $\sigma$-algebra $\mathcal{B}_{\mathbb{X}}$  \citep{dalay2008book}. For the space $\mathbb{S}=\mathbb{R}^+ \times \mathbb{X}$, we  denote its Borel $\sigma$-algebra by $\mathcal{S}$.

\begin{definition}\label{def: CRM} 
Let $\tilde{\mu}$ be a measurable function 
 from $(\Omega, \mathcal{F}, \mathbb{P})$ into $(\mathbb{B}_{\mathbb{X}}, \mathcal{B}_{\mathbb{X}})$. If the random variables $\tilde{\mu}(A_1), \cdots, \tilde{\mu}(A_d)$ are mutually independent, for any pairwise disjoint sets $A_1,\cdots,A_d$ in $\mathcal{X}$ and any finite integer $d\geq 2$. Then, $\tilde{\mu}$ is termed a completely random measure (CRM).
\end{definition}

As is shown in \citep{kingman1967completely, kingman1975}, each CRM can be decomposed into three terms: a deterministic term, an atomic term with deterministic atoms, and an atomic  term with no fixed atoms. In this paper, we focus on  CRMs that are  almost surely discrete measures with probability one.   This essentially entails the discreteness of random measures obtained as transformations of CRMs. To construct such CRMs, one usually uses the functional of Poisson random measures defined as follows.   A Poisson random measure $\tilde{N}$ on  $\mathbb{S}$   with 
 mean measure $\nu(ds,dx)$ is a random measure  from   $ \Omega\times \mathbb{S} $ to $\mathbb{R}^+$   satisfying:
\begin{enumerate}
\item[(i)] $\tilde{N}(B) \sim \text{Poisson}(\nu(B))$ for any $B$ in $\mathcal{S}$ such that $\nu(B)<\infty$; 
\item[(ii)] for any pairwise disjoint sets $B_1,\cdots,B_m$ in $\mathcal{S}$, the random variables $\tilde{N}(B_1), \cdots, \\
\tilde{N}(B_m)$ are mutually independent.  
\end{enumerate}
The Poisson mean measure $\nu$  satisfies  the condition (see  \citep{dalay2008book} for details of Poisson random measures) that 
\begin{align}
\int_0^{\infty}\int_{\mathbb{X}} \min (s,1)\nu(ds,dx) <\infty\,.\label{condition.finite}
\end{align}
Let  $\tilde{\mu}$ be a random measure defined on $(\Omega, \mathcal{F}, \mathbb{P})$ that takes values in $(\mathbb{B}_{\mathbb{X}},\mathcal{B}_{\mathbb{X}})$ defined as follows.
\begin{align}
\tilde{\mu}(A):=\int_0^{\infty}\int_A s \tilde{N}(ds, dx), \quad \forall A \in \mathcal{X}\,. \label{random measure definition}
\end{align}
It is straightforward to verify that $\tilde{\mu}$ is a completely random measure.  It is also known that  $\tilde{\mu}$  is  uniquely characterized by its Laplace transform \citep{kingman1967completely, kingman1975} as follows. For any $A \in \mathcal{X}$ and any measurable function $h: \mathbb{X} \rightarrow \mathbb{R}$ such that $\int |h| d\tilde{\mu}< \infty$, we have
\begin{align}
\mathbb{E}\left[e^{-\int_A h(x) \tilde{\mu}(dx)}\right]=\exp \left\{-\int_0^{\infty}\int_A\left[ 1-e^{- s h(x)}\right]\nu(ds, dx)\right\}\,.   \label{eq: laplace}
\end{align}
The measure $\nu$ characterizing $\tilde{\mu}$ is termed as   the \textit{L{\'e}vy intensity} of $\tilde{\mu}$.
 
As commonly assumed in many studies and throughout this paper, we shall assume that the L{\'e}vy intensity can be factorized as 
\begin{align}
\nu(ds, dx)=\rho(ds)\alpha(dx)\,,\label{eq: homogeneous}
\end{align}
 where $\rho: \mathcal{B}(\mathbb{R}^+)\to  \mathbb{R}^+$ is some measure on $\mathbb{R}^+$ and $\alpha$ is a non-atomic measure on $(\mathbb{X},\mathcal{X})$ so  that $\alpha(\mathbb{X})=a <\infty$. The total mass $a=\alpha(\mathbb{X})$ is called the {\it concentration parameter}. Usually,  the finite measure $\alpha$ is written in the form $\alpha(dx)=aH(dx)$ for some   probability measure $H$ and the concentration parameter $a \in (0, \infty)$. A CRM with L{\'e}vy intensity satisfying \eqref{eq: homogeneous} is called a {\it homogeneous} CRM and will be denoted as $\tilde{\mu}_a$ hereafter.

To construct hNRMIs,   the above   homogeneous CRM is normalized, and thus one needs  $\tilde{\mu}_a(\mathbb{X})$ to be finite and positive almost surely. The finiteness is guaranteed by the condition \eqref{condition.finite} while the positivity follows from the condition  $\rho(\mathbb{R}^+)=\infty$ (see p.563 in \citet{Regazzini_2003}).   Based on  the above constructions, an  hNRMI $P$ on $(\mathbb{X},\mathcal{X})$ is a random probability measure defined by
\begin{align}
P(\cdot)=\frac{\tilde{\mu}_a(\cdot)}{\tilde{\mu}_a(\mathbb{X})}\,, \label{eq: hNRMIs}
\end{align}
where we omit $a$ for $P$ for notional simplicity. Moreover, $P$ is discrete almost surely due to the discreteness of $\tilde{\mu}_a$. The measure $\rho(ds)$ in the L{\'e}vy intensity characterizes the corresponding hNRMI. For instance, if $\rho(ds)=\frac{e^{-s}}{s}ds$, $P$ is the Dirichlet process; if $\rho(ds)=\sum_{i=1}^{\gamma}\frac{e^{-i s}}{s}ds$ with some positive integer $\gamma$, $P$ is the generalized Dirichlet process; if $\rho(ds)=\frac{\sigma}{\Gamma(1-\sigma)s^{1+\sigma}}ds$ with $\sigma \in (0,1)$, $P$ is the $\sigma$-stable process; if $\rho(dx)=\frac{e^{-\frac{1}{2} s}}{\sqrt{2\pi}s^{3/2}}ds$, $P$ is the normalized inverse Gaussian process; more generally, if $\rho(ds)=\frac{\sigma e^{-\theta s}}{\Gamma(1-\sigma)s^{1+\sigma}}ds$ with $\sigma \in (0,1)$, $\theta >0$, $P$ is the normalized generalized gamma process.
\subsection{Representation  of CRMs}\label{subsec.decomposition}
When the concentration parameter $a$ approaches infinity, it is trivial to see the expectation of a homogeneous CRM $\tilde{\mu}_a$ with L{\'e}vy intensity $\nu(ds,dx)=\rho(ds)aH(dx)$ also approaches  infinity as follows. For any $B \in \mathcal{X}$ such that $H(B)>0$,
\begin{align}
\E \lk \tilde{\mu}_a(B) \rk&=-\frac{d}{d \lambda}\E\lk e^{-\lambda \tilde{\mu}_a(B)} \rk \bigg|_{\lambda=0}\nonumber\\
&=-\frac{d}{d \lambda} \exp \ltd -aH(B)\int_0^{\infty} \lc 1-e^{-\lambda s} \rc\rho(ds) \rtd \bigg|_{\lambda=0}=aH(B)\tau_1\,,
\end{align}
where and in the remaining part of this work $\tau_k=\int_0^{\infty} s^k \rho(ds)$ for any positive integer $k$. Intuitively, when the concentration parameter $a \to \infty$, an hNRMI should converges to $H$. To show this convergence, we   introduce the following representation  of a homogeneous CRM. 

Let $N(d\zeta, ds,dx)$ be a Poisson random measures on $[0,\infty) \times \mathbb{S}$ with finite mean measure $\nu_0(d\zeta, ds,dx)=d\zeta \rho(ds)H(dx)$ such that $\int_0^{\infty} \min (s,1) \rho(ds)<\infty$ and $\rho(\mathbb{R}^+)=\infty$. For any $A \in \mathcal{X}$, define
\begin{align}
\tilde{\mu}_a(A):=\int_0^a \int_0^{\infty} \int_A sN(d\zeta, ds,dx)\,.\label{e.2.7} 
\end{align}
Therefore, for any measurable function $h: \mathbb{X}\rightarrow \mathbb{R}$ such that $\int |h| d\tilde{\mu}_a <\infty$, we have
\begin{align*}
\mathbb{E}\left[e^{-\int_A h(x) \tilde{\mu}_a(dx)}\right]=
\exp \left\{-a\int_0^{\infty}\int_A\left[ 1-e^{- s h(x)}\right]\rho(ds)H(dx)\right\}\,.
\end{align*}
This  $\tilde{\mu}_a$ has the L{\'e}vy intensity $\nu(ds,dx)=\rho(ds)aH(dx)$. In the following we shall  study the limit of  $\tilde \mu_a$   
defined by \eqref{e.2.7}.  It is easy to see $\tilde{\mu}_0$ is $0$. On the other hand, due to the fact that $N(B_1)$ is independent of $N(B_2)$ for disjoint $B_1$ and $B_2$ in the Borel $\sigma$-algebra of $[0,\infty)\times \mathbb{S}$,   $\tilde{\mu}_b-\tilde{\mu}_a$ is independent of $\tilde{\mu}_a$ for any $b>a$. 
Based on the above construction of $\tilde{\mu}_a$, one has the following decomposition.
\begin{align}
\tilde{\mu}_a=\sum_{m=1}^{\lfloor a \rfloor} \mu_m + \mu_{\Delta}\,,\label{e.2.8} 
\end{align}
where $\lfloor a \rfloor$ is the largest integer that is smaller than or equal to $a$ and we denote $\mu_m=\tilde{\mu}_m -\tilde{\mu}_{m-1}$, $\mu_{\Delta}=\tilde{\mu}_a-\tilde{\mu}_{\lfloor a \rfloor}$ with $\Delta=a-\lfloor a \rfloor \in [0,1]$. Therefore, $\{\mu_m\}_{m=1}^{\lfloor a \rfloor}$ is a sequence of independent and identically distributed CRMs that is independent of the CRM $\mu_{\Delta}$.
The reasons of introducing the above decomposition are the follows: i) $\tilde{\mu}_a$ can be written as a summation of independent CRMs for any $a>0$; ii) both $\mu_m$ and $\mu_{\Delta}$ have finite first moments when $\tau_1<\infty$ and finite second moments when  $\tau_2<\infty$. For instance, for any $A\in \mathcal{X}$ and any $m \in (1,\cdots,\lfloor a \rfloor)$,
\begin{align*}
&\E\lk \mu_m(A) \rk=\E\lk \lc \tilde{\mu}_m -\tilde{\mu}_{m-1}\rc \rk= -\frac{d}{d \lambda} \exp \ltd -H(A)\int_0^{\infty} \lc 1-e^{-\lambda s} \rc \rtd \bigg|_{\lambda=0}=H(A)\tau_1\,,\\
&\E\lk \mu_m(A)^2 \rk=\E\lk \lc \tilde{\mu}_m -\tilde{\mu}_{m-1}\rc^2 \rk\\
&\qquad\qquad\quad=\frac{d^2}{d \lambda^2} \exp \ltd -H(A)\int_0^{\infty} \lc 1-e^{-\lambda s} \rc \rtd \bigg|_{\lambda=0}=H(A)^2\tau_1^2+H(A)\tau_2\,.\\
&\E\lk \mu_{\Delta}(A) \rk=\E\lk \lc \tilde{\mu}_a-\tilde{\mu}_{\lfloor a \rfloor}\rc \rk= -\frac{d}{d \lambda} \exp \ltd -\Delta H(A)\int_0^{\infty} \lc 1-e^{-\lambda s} \rc \rtd \bigg|_{\lambda=0}=\Delta H(A)\tau_1\,,\\
&\E\lk \mu_{\Delta}(A)^2 \rk=\E\lk \lc \tilde{\mu}_a-\tilde{\mu}_{\lfloor a \rfloor}\rc^2 \rk\\
&\qquad\qquad\quad=\frac{d^2}{d \lambda^2} \exp \ltd -H(A)\int_0^{\infty} \lc 1-e^{-\lambda s} \rc \rtd \bigg|_{\lambda=0}=\Delta^2 H(A)^2\tau_1^2+\Delta H(A)\tau_2\,.
\end{align*}

Moreover, let  
\[
\tilde{\tau}_k=\int_1^{\infty} s^k \rho(ds)
\quad \hbox{ for any $k \in \mathbb{Z}^+$}
\]
 throughout the paper. It is worth pointing out that the assumption $\tilde{\tau}_k< \infty$  implies $\tau_k<\infty$ for each $k \in \mathbb{Z}^+$, due to the fact that $\int_0^{\infty} \min(s,1)\rho(ds)<\infty$.  
 From \eqref{e.2.8}  one can obtain    the following proposition for the CRM $\tilde{\mu}_M$ with a positive integer concentration parameter $M$.
\begin{proposition}\label{prop.CRM}
Let $A \in \mathcal{X}$ and let $(A_1, \cdots, A_n)$ be disjoint measurable subsets belonging  to $\mathcal{X}$ such that $H(A_i)\in (0,1)$ for all $i \in \{1,\cdots,n\}$. Denote the random vector $\bZ:=(Z_1,\cdots,Z_n)\sim N(0,\Sigma)$ with $\Sigma=(\sigma_{ij})_{1\leq i,j \leq n}$ being  given by 
\begin{align}
    \sigma_{ij}= \begin{cases}
    1 & \qquad \mbox{if $i=j$}\,,  \\
    -\sqrt{\frac{H(A_i)H(A_j)}{\left(1-H(A_i)\right)\left(1-H(A_j)\right)}}& \qquad \mbox{if $i \neq j$}\,.  \\
\end{cases} \label{eq: sigma}
\end{align}
\begin{itemize}
\item[(i)] For any positive integer $p$, we have 
\begin{align}
\E\lk \lc\mu_m(A) \rc^p \rk=\sum_{i=1}^p  H(A)^i \frac{1}{i!} \sum_{*}  {p \choose q_1, \cdots, q_i }\prod_{j=1}^i  \tau_{q_j}\,,
\end{align}
where the summation $\sum_{*}$ is taken over all vectors $(q_1,\cdots, q_i)$ of positive integers such that $\sum_{j=1}^i q_j=p$.
\item[(ii)]  $\tilde{\mu}_{\lfloor a \rfloor}(A)\leq \tilde{\mu}_a(A) \leq \tilde{\mu}_{\lceil a \rceil}(A)$ a.s., where $\lceil a \rceil$ is the smallest integer that is larger than or equal to $a$.
\item[(iii)] If $\tilde{\tau}_1 < \infty$, then $\lim_{M \rightarrow \infty} \frac{\tilde{\mu}_M(A)}{M}=H(A)\tau_1$ a.s..
\item[(iv)] If $\tilde{\tau}_2< \infty$, then $(R_M(A_1), \cdots, R_M(A_n)) \overset{d}{\rightarrow} \bZ$,  where $R_M(\cdot)=\frac{\tilde{\mu}_M(\cdot)-H(\cdot)\tilde{\mu}_M(\mathbb{X})}{\sqrt{M\tau_2H(\cdot)(1-H(\cdot))}}$.
\end{itemize}

\end{proposition}
Result (i) in Proposition \ref{prop.CRM} indicates that the existence of the $p$-th moment of $\mu_m$ requires the finiteness of $\tau_p$ (or $\tilde{\tau}_p$). This provides the implication of the requirement of assumptions in (iii) and (iv). The proof of (ii) in the above proposition is due to the non-negativity of Poisson random measure. The results (iii) and (iv) are mainly relying on the decomposition $\tilde{\mu}_M(\cdot)=\sum_{m=1}^M \mu_m(\cdot)$. Again, assumptions $\tilde{\tau}_1 < \infty$ in result (ii) and  $\tilde{\tau}_2< \infty$ in result (iii) guarantee the finiteness of the first and second moments of $\mu_m$ for all $m \in (1,\cdots,M)$. It is worth pointing out that assumptions $\tilde{\tau}_p<\infty$ can be easily checked, as $\rho$ will be explicitly given when an hNRMI is defined (see the particular cases of hNRMIs below \eqref{eq: hNRMIs}).  Although Proposition \ref{prop.CRM} only provides the asymptotic behaviours of homogeneous CRM with an integer concentration parameter, it plays an important role in the proofs of the main results in the next section.

\section{Main results}\label{sec.main results}
In this section, we shall present multiple asymptotic theorems for hNRMIs when the concentration parameter $a \rightarrow \infty$. 
To see how the concentration parameter $a$ relates  to hNRMIs, we  state in the next proposition that the variance of hNRMIs is of order $\frac{1}{a}$ when $a$ is large.
\begin{proposition}\label{prop: order of var}
Let $P$ be an hNRMI satisfying $\tilde{\tau}_2<\infty$. For any $A\in \mathcal{A}$,
\begin{align*}
&\E[P(A)]=H(A)\,,\\
 &\Var \lc P(A) \rc=H(A)(1-H(A))\mathcal{I}_a\,,
 \end{align*}
 where 
 \begin{align}
\mathcal{I}_a=a \int_0^{\infty} \lambda \exp\ltd -a\int_0^{\infty}\lc 1-e^{-\lambda s} \rc \rho(ds) \rtd \int_0^{\infty} s^2 e^{-\lambda s}\rho(ds) d\lambda\,.\nonumber
\end{align}
  Moreover, we have $\mathcal{I}_a=O\lc \frac{1}{a}\rc $ for large $a$.
\end{proposition}

\subsection{Strong law of large numbers}
In this subsection we discuss the strong law of large numbers and the Glivenko-Cantelli theorem, namely, we want to show    the  converges of hNRMIs to its mean measure almost surely. 

Recall that $(\mathbb{M}_{\mathbb{X}}, \mathcal{M}_{\mathbb{X}})$ is the space of all probability measures on $\mathbb{X}$. To introduce the convergence of a random probability measure $P$ on $\mathbb{M}_{\mathbb{X}}$, let $\eta$ be the metric on $\mathbb{X}$ and $C_b(\mathbb{X},\eta)$ be the collection of bounded continuous functions on metric space $(\mathbb{X},\eta)$. It is well-known that when the metric space $(\mathbb{X},\eta)$ is compact, $C_b(\mathbb{X},\eta)$ equipped with 	the uniform convergence topology has a countable dense subset $\{f_i \in C_b(\mathbb{X},\eta): i\geq 1\}$. In the noncompact case when $\mathbb{X}=\mathbb{R}$, the existence of a countable dense subset $\{f_i \in C_b(\mathbb{X},\eta): i\geq 1\}$ is guaranteed when $C_b(\mathbb{X},\eta)$ is equipped with the metric $\sum_{i=1}^{\infty} 2^{-i} \frac{\sup_{x\in [0,i]} |f(x)-g(x)|}{1+\sup_{x\in [0,i]} |f(x)-g(x)|}$ for $f,g \in C_b(\mathbb{X},\eta)$. For the general Polish space $\mathbb{X}$, the space $\mathbb{X}$ is Borel isomorphic to $\mathbb{R}$ by Theorem 8.3.6 in \citet{cohn2013measure}. Therefore, there exists a metric $\tilde{\eta}$ on $\mathbb{X}$ such that all the Borel sets in $(\mathbb{X},\tilde{\eta})$ are the same as the original topology on $\mathbb{X}$. Thus, $C_b(\mathbb{X},\tilde{\eta})$ has a countable dense subset $\{f_i \in C_b(\mathbb{X},\eta): i\geq 1\}$. 

 Let $\mathbb{M}_{\mathbb{X}}$ be equipped with the weak topology and the metric 
\begin{align}
d(\xi_1, \xi_2)=\sum_{i=1}^{\infty}\frac{|\langle \xi_1-\xi_2, f_i \rangle|\wedge 1}{2^i}\quad \text{for } \xi_1,\xi_2 \in \mathbb{M}_{\mathbb{X}}\,, \label{eq: metric}
\end{align}
where $\{f_i \in C_b(\mathbb{X},\eta): i\geq 1\}$ is a  countable dense subset of $C_b(\mathbb{X},\eta)$ and $\langle \xi, f \rangle=\int f d\xi$. The existence of a  countable dense subset $\{f_i \in C_b(\mathbb{X},\eta): i\geq 1\}$ of $C_b(\mathbb{X},\eta)$ is a necessity when to define  the above metric.

\begin{theorem}\label{thm: slln}
Let $P$ be an hNRMI satisfying $\tilde{\tau}_1<\infty$. When $a \rightarrow \infty$, we have the following results.
\begin{itemize}
\item[(i)] (Strong law of large numbers) The random probability measure $P$ converges almost surely to its mean measure $H$ in $\mathbb{M}_{\mathbb{X}}$, i.e.,
\begin{align}
P \overset{a.s.}{\rightarrow} H\,, \label{eq: slln}
\end{align}
with respect to the distance $d$ defined by \eqref{eq: metric}. 
\item[(2)] (Glivenko-Cantelli strong law of large numbers) Let $(\mathbb{X},\mathcal{X})=(\mathbb{R},\mathcal{B}(\mathbb{R}))$ be the Euclidean space, we have
\begin{align}
\sup_{x\in\mathbb{R}}|P\lc(-\infty,x])-H((-\infty,x]\rc| \overset{a.s.}{\rightarrow} 0\,. \label{eq: glivenko slln}
\end{align}
\end{itemize}
\end{theorem}

Theorem \ref{thm: slln} provides a strong convergent result of hNRMIs as random probability measures when their concentration parameter is large. It  includes the strong law of large numbers for some particular cases of hNRMIs studied in prior works.
\begin{example}
Theorem \ref{thm: slln} includes the following existing results:
\begin{itemize}
\item[(i)] Theorem 2 in \citet{labadi2016}: $P(A) \overset{a.s.}{\to} H(A)$ for any $A \in \mathcal{X}$  when $P$ is the normalized inverse Gaussian process. The convergence is for $a=N^2$ and $N \to \infty$.
\item[(ii)] Theorem 4.2 in \citet{huzhang2022functional}: $P(A) \overset{a.s.}{\to} H(A)$ for any $A \in \mathcal{X}$  when $P$ is one of the Dirichlet process, the Generalized Dirichlet process, the normalized generalized gamma process, the normalized inverse Gaussian process. The convergence is for  $a=N^{\kappa}$ with any fixed $\kappa>0$ and $N \to \infty$.
\end{itemize}
\end{example}
It is worth pointing out that in Theorem \ref{thm: slln}, there is no assumption of the convergent rate of $a$ and the convergence is in the measure space $\mathbb{M}_{\mathbb{X}}$. This provides more applicability of our results.

\subsection{CLT and functional CLT for hNRMIs}\label{subsec: CLT}
Having establishing  the strong law of large numbers, we now consider the central limit theorem and the  functional central limit theorem for hNRMIs when $a \rightarrow \infty$. To make the presentation concise we introduce the following centered and scaled random measures of $P$.
\begin{align}
&D_a(\cdot)=\frac{\sqrt{a} \tau_1 \lc P(\cdot)-H(\cdot)\rc}{\sqrt{ H(\cdot)(1-H(\cdot))  \tau_2}}\,, \label{eq: notation D}\\
&Q_{H,a}=\frac{\sqrt{a} \tau_1 \lc P(\cdot)-H(\cdot)\rc}{\sqrt{  \tau_2}}\,. \label{eq: notation Q}
\end{align}

Let $\mathbb{B}_{H}^o$ be the Brownian bridge with parameter $H$ or $H$-Brownian bridge (see e.g., \citep{kim2003}), namely, $\mathbb{B}_{H}^o$ is a Gaussian process with $\E[\mathbb{B}_{H}^o(A)]=0$ and $\E[\mathbb{B}_{H}^o(A_1)\mathbb{B}_{H}^o(A_2)]=H(A_1 \cap A_2)-H(A_1)H(A_2)$ for any $A_1,A_2 \in \mathcal{X}$. 
  With these  notations, we are ready to give the central limit theorem and the functional central limit theorem as follows.

\begin{theorem}\label{thm: clt of NRMIs}
Let $P$ be an  hNRMI satisfying  $\tilde{\tau}_2<\infty$. When $a \rightarrow \infty$, we have the following results.
\begin{itemize}
\item[(i)] (Central limit theorem)
For any disjoint measurable subset $(A_1,\cdots, A_n)$ in $\mathcal{X}$ such that $H(A_i)\in (0,1)$ for all $i \in \{1,\cdots,n\}$, we have
\begin{equation}
(D_a(A_1), \cdots, D_a(A_n)) \overset{d}{\rightarrow} \bZ\,,\label{eq: clt of D}
\end{equation}
where $\bZ\sim N(0,\Sigma)$ with $\Sigma=(\sigma_{ij})_{1\leq i,j \leq n}$ being  given by 
\begin{align}
    \sigma_{ij}= \begin{cases}
    1 & \qquad \mbox{if $i=j$}\,,  \\
    -\sqrt{\frac{H(A_i)H(A_j)}{\left(1-H(A_i)\right)\left(1-H(A_j)\right)}}& \qquad \mbox{if $i \neq j$}\,.  \\
\end{cases} \label{eq: sigma}
\end{align}
\item[(ii)](Functional central limit theorem)
The random measure $Q_{H,a}$ converges weakly to $\mathbb{B}_H^o$ in $\mathbb{M}_{\mathbb{X}}$ with respect to the metric $d$ in \eqref{eq: metric}, i.e.,
\begin{align}
Q_{H,a} \overset{\rm weakly}{\rightarrow} \quad \mathbb{B}_H^o\,. \label{eq: clt of D}
\end{align}
In particular, if $(\mathbb{X}, \mathcal{X})=(\mathbb{R}^d, \mathcal{B}(\RR^d))$ is the $d$-dimensional Euclidean space
with the Borel $\si$-algebra. Then 
\begin{equation}
Q_{H,a}  \stackrel{{\rm weakly}}{\rightarrow }  \quad {\mathbb{B}_H^o}\qquad  {\rm in}\qquad D(\mathbb{R} ^d)\  \label{eq: functional clt} 
\end{equation}  
with respect to the Skorohod topology. 
\end{itemize}
\end{theorem}

In result (i) in  Theorem \ref{thm: clt of NRMIs}, if $H(A_j)=0$ for some $A_j$, we shall drop the corresponding component $D_a(A_j)$, since in this case $P(A_j) \overset{p}{\to} 0$. 
The proof of Theorem \ref{thm: clt of NRMIs} mainly relies on the decomposition as we introduced in Section \ref{subsec.decomposition}. The details of the proof and the introduction of the weak convergence in \eqref{eq: functional clt} in space $D(\mathbb{R} ^d)$ are given in the supplementary material. In the next example, one can see Theorem \ref{thm: clt of NRMIs} includes the asymptotic results of multiple particular Bayesian nonparametric priors studied in prior works.
\begin{example}
Theorem \ref{thm: clt of NRMIs} includes the following known cases:
\begin{itemize}
\item[(i)] Theorem 8 in \citet{labadi2016} where $P$ is the Dirichlet process and $\mathbb{X}=\mathbb{R}$.
\item[(ii)] Theorem 1 in \citet{labadi2013} where $P$ is the normalized inverse Gaussian process and $\mathbb{X}=\mathbb{R}$.
\item[(iii)] Theorem 4.15 in \citet{huzhang2022functional} where $P$ is one of the Dirichlet process, the Generalized Dirichlet process, the generalized normalized gamma process and the normalized inverse Gaussian process.
\end{itemize}

\end{example}
The methods of deriving the central limit theorems used in \citet{labadi2013, labadi2016} and in \citet{huzhang2022functional} can not be generalized in our case when $P$ is a general  hNRMI. In fact, \citet{labadi2013, labadi2016} derive the central limit theorems through the marginal densities of the Dirichlet process and the normalized inverse Gaussian, however, the marginal densities of a general hNRMI  is  unknown. \citet{huzhang2022functional} prove  the central limit theorems through the stick-breaking representations of multiple Bayesian nonparametric priors by using the method of moments, whereas the moments of hNRMIs have complicated forms (see e.g. \citep{james2006conjugacy}) that are hard to  apply  in our case.  Although the stick-breaking representation for general hNRMI is given by Proposition 1 in \citet{favaro2016}, the method in \citet{huzhang2022functional} is not easy to apply in general,  as the stick-breaking weights $(v_i)_{i\geq 1}$ are dependent  and have complicate joint marginal distributions. Thus, the results in Theorem \ref{thm: clt of NRMIs} complements that of  \citet{huzhang2022functional} and provide a new way to study the asymptotic behaviour of Bayesian nonparametric priors.

\begin{remark}
The Dirichlet process is a well known special case of hNRMIs. Since the posterior process of the Dirichlet process is also a Dirichlet process, it follows from Theorem \ref{thm: clt of NRMIs} that  the central limit theorem and the functional central limit theorem hold for the posterior of the Dirichlet process in the following cases: (i) with large sample size and finite $a$; (ii) with large $a$ and finite sample size; (iii) with $a$ and sample size both large.
\end{remark}

One direct consequence  of the central limit theorem is the application of the  delta method. The following result for hNRMIs can be obtained from  Theorem \ref{thm: clt of NRMIs} and Theorem 3.10.4 in \citet{Vaart2023}.

\begin{corollary}\label{thm: delta method}
Let $P$ be an hNRMI satisfying $\tilde{\tau}_2<\infty$. Let  $\mathbb{D}$ be the metric space of all probability measures
on $(\mathbb{X},  \mathcal{X})$  with the total variation  distance.   
Let 
 $\phi :  \mathbb{D}  \rightarrow \mathbb{R}^d$ be a continuous functional which is   Hadamard differentiable on $\mathbb{D}$.
 Then, as $a \rightarrow \infty$, we have 
 \[
\frac{\sqrt{a}\tau_1}{\sqrt{\tau_2}}\left(\phi\left(P(\cdot)\right)-\phi\left(H(\cdot)\right)\right)\overset{weakly}{\rightarrow}\phi'_{H(\cdot)}\left(B_H^o\right)\,. 
\]
\end{corollary}
To apply the delta method again in  Corollary \ref{thm: delta method}, we can find the limiting distribution of the quantile process of hNRMIs in the following example.

\begin{example}
Let $P$ be an hNRMI on $(\mathbb{X}, \mathcal{X})$ satisfying $\tilde{\tau}_2<\infty$. 
Let   $H$  be absolutely  continuous  
 with positive derivative $h$.  
By   Example 3.10.22 of \citet{Vaart2023}, we have  
\begin{align}
\frac{\sqrt{a}\tau_1}{\sqrt{\tau_2}}\left(P^{-1}(\cdot)-H^{-1}(\cdot)\right)\overset{weakly}{\rightarrow}-\frac{B^o(\cdot)}{h\left(H^{-1}(\cdot)\right)}=G(\cdot)\,,  
\end{align} 
where $H^{-1}(s)=\inf \{t: H(t)\geq s\}$. 
  The limiting process $G$ is a Gaussian process with zero-mean and 
with covariance function 
\begin{align*}
\Cov\left(G\left((0,s]\right),G\left((0,t]\right)\right)=\frac{s \wedge t-st}{h\left(H^{-1}\left((0,s]\right)\right)h\left(H^{-1}\left((0,t]\right)\right)} 
\end{align*}
for $s,t \in \mathbb{R}$,
\end{example}
Moreover, one can obtain the conditional delta method of $P$ by using Theorem \ref{thm: clt of NRMIs} and Theorems 3.10.13-3.10.15 in \citet{Vaart2023}.

\subsection{Berry-Esseen bound}\label{subsec: Berry-Esseen}

To see the convergence rate of the central limit theorem in Theorem \ref{thm: clt of NRMIs}, we obtain the Berry-Esseen bound in the next theorem. We shall first introduce the following useful notations.
Let $F_{\bX}(\bt)$ denote the distribution function of the finite dimensional random vector $\bX$. Denote $\mathbf{D}_a:=(D_a(A_1),\cdots, D_a(A_n))$ for any disjoint $A_1, \cdots, A_n$ in $\mathcal{X}$ such that $H(A_i)\in (0,1)$ fir all $i \in \{1,\cdots,n\}$. Recall that $\bZ:=(Z_1,\cdots, Z_n)$ is the $n$ dimensional normal random vector as defined in (i) of Theorem \ref{thm: clt of NRMIs}.  

\begin{theorem}\label{thm: berry-esseen bound}
Let $P$ be an hNRMI satisfying $\tilde{\tau}_3<\infty$. When $a$ is large, we have the following results.
\begin{itemize}
\item[(i)] For any $f \in C_b(\mathbb{X}, \eta)$, there exist a positive constant $c_1$ such that 
\begin{align}
\sup_{t \in \mathbb{R}} |F_{\langle Q_{H,a}, f \rangle} (t) - F_{\langle \mathbb{B}_H^o, f \rangle} (t)| \leq \frac{c_1}{\sqrt{a}}\,, \label{eq: functional berry-esseen}
\end{align}
where $c_1$ depends on $H$, $\tau_2$ and $\tau_3$.
\item[(ii)] For any disjoint measurable subsets $(A_1, \cdots, A_n)$ of $\mathbb{X}$, there exist a positive constant $c_2$ such that
\begin{align}
\sup_{\bt \in \mathbb{R}^n} | F_{ \mathbf{D}_a}(\bt) -F_{\bZ}(\bt)| \leq \frac{c_2}{\sqrt{a}}\,,\label{eq: berry-esseen}
\end{align}
where $c_2$ depends on the dimension $n$, $H$, $\tau_2$ and $\tau_3$.
\end{itemize}

\end{theorem}

To the best of  our knowledge, Theorem \ref{thm: berry-esseen bound} is the first work to present the Berry-Esseen bound in asymptotic studies of Bayesian nonparametric models. The theorem   states that the convergence rate of $Q_{H,a}$ to $\mathbb{B}_H^o$ and $\mathbf{D}_a$ to $\bZ$ are both $O \lc \frac{1}{\sqrt{a}} \rc$. Result (ii) in Theorem \ref{thm: berry-esseen bound} does not consider the case when the dimension $n$ is very large (in order of $a$), this high-dimensional case can be also derived by using the results in \citet{kuchibhotla2020high}.

\section{Numerical illustration}\label{sec. numerical}

To illustrate the central limit convergence in Theorem \ref{thm: clt of NRMIs} when $a \rightarrow \infty$, we   carry out some  numerical simulations to visualize  how the convergence behaves. To avoid the effect from other model parameters apart from the concentration parameter, we will run the simulations of the Dirichlet  processes and the normalized inverse Gaussian processes   for different choices of $a$.

In our simulation, we take $(\mathbb{X}, \mathcal{X})=((0,1), \mathcal{B}((0,1)))$ and let $H$ be the uniform measure on $(0,1)$. Without loss of generality, we consider the case when $n=3$ and the partition of $\mathbb{X}=(0,1)$ is given as $A_1=(0, 0.3], A_2=(0.3, 0.7], A_3=(0.7, 1)$. In our simulations of the two processes, we use the stick-breaking representations of the Dirichlet process \citep{sethuraman1994} and of the normalized inverse Gaussian process \citep{favaro2012} (see also the stick-breaking representation of hNRMIs in \citet{favaro2016}). We truncate the infinite summation of their stick-breaking representations at $3000$ to make sure the truncation error is much less than $\frac{1}{a^2}$ and generate $3000$ samples of $(D_a(A_1), D_a(A_2), D_a(A_3))$ for each $a=2, 5, 10, 20, 30$. Instead of visualizing the messy densities of $(D_a(A_1), D_a(A_2), D_a(A_3))$, we plot the densities of the linear combination $1.6D_a(A_1)+1.49D_a(A_2)+0.50D_a(A_3)$ (any other linear combinations will produce the similar results with a different variance)  in Figure \ref{pdpplot} for the Dirichlet  processes (left) and the normalized inverse Gaussian processes  (right) correspondingly. It is apparent  to observe  the convergence to  normal shape  when $a$ is getting large.

\begin{figure}
\centering     
\subfigure{\includegraphics[width=0.49\textwidth , height=60mm]{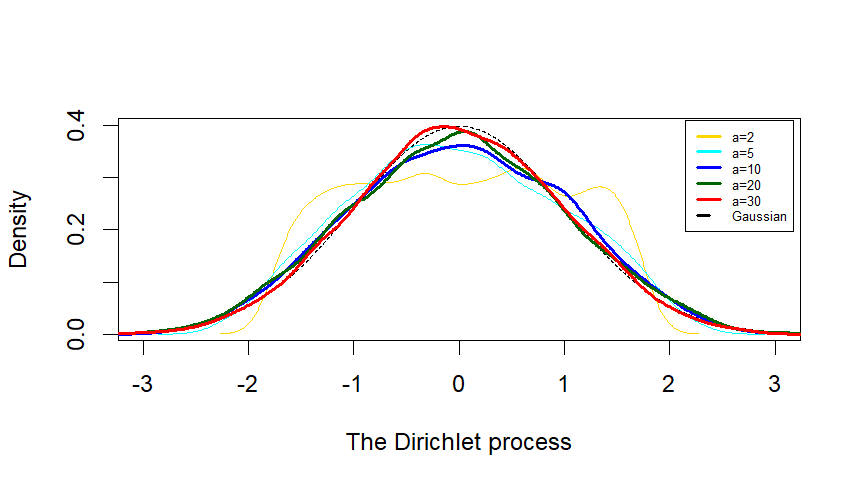}}
\subfigure{\includegraphics[width=0.49\textwidth , height=60mm]{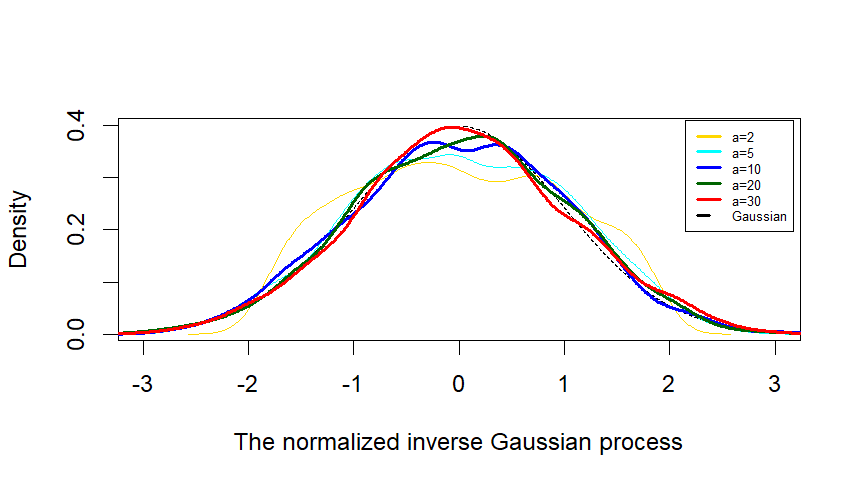}}
\caption{The convergence of $\mathbf{D}_a$ for the Dirichlet process and the normalized inverse Gaussian process when $a=2,5,10,20,30$.}\label{pdpplot}
\end{figure}
\section{Discussion}\label{sec.discussion}

We display the strong law of large numbers, the central limit theorem, the functional central limit theorem and the Berry-Esseen bound of hNRMIs when the concentration parameter $a \rightarrow \infty$. These results together with the large sample asymptotic analysis in \citet{junxi2023large} present a complete asymptotic analysis of NRMIs. Another important Bayesian nonparametric prior is  the  two-parameter Poisson-Dirichlet process (also known as the Pitman-Yor process) \citep{pitman1997, perman1992, ghosal2017fundamentals}, which is not an hNRMI but close related to hNRMIs. The asymptotic analysis of the two-parameter Poisson-Dirichlet process is given in \citet{huzhang2022functional, dawson2006} when $a \rightarrow \infty$ and in \citet{james2008, franssen2022bernstein} when the sample size is large. There are more interesting open problems along the asymptotic analysis path in Bayesian nonparametrics. For example, large sample and large concentration parameter asymptotic analysis of other general classes of Bayesian nonparametric priors, like species sampling priors \citep{pitman1996some, aldous1985exchangeability} and Gibbs-type priors \citep{de2015gibbs, gnedin2006exchangeable}.
\section*{Acknowledgments}
\bibliographystyle{Chicago}
\bibliography{huzhang}
\newpage
\section{Supplementary materials}\label{sec: supp}
In this supplementary material, we     recall some useful definitions
and include the proofs of the main results of  the paper.

\subsection{Introduction of $D(\mathbb{R} ^d)$}
To state the functional central limit 
theorem in Theorem \ref{thm: clt of NRMIs},  we  need the space $D(\RR^d)$ introduced in
 Section 3 of \citep{bickel1971}. The  characteristics  
  of    the elements (functions)  in $D(\RR^d)$ are given by  their continuity properties 
  described as follows. For   $1 \leq p\leq d$, let  $R_p$ be one of the relations $<$ or  $\geq$  and for $t 
  =(t_1, \cdots, t_d)\in \RR^d$  let $\mathcal{Q}_{R_1,\cdots,R_d}$ be the quadrant
\[
\mathcal{Q}_{R_1,\cdots,R_d}:=\left\{(s_1,\cdots,s_d)\in \RR^d :\, s_p R_p t_p,\, 1\leq p \leq d\right\}.
\]
Then, $x \in D(\RR^d)$ if and only if 
 (see e.g., \citep{straf1972}) for each $t\in \RR^d$,
the following two conditions hold:
 (i) $x_{\mathcal{Q}}=\lim_{s \rightarrow t,\, s \in \mathcal{Q}}x(s)$ exists for each of the $2^d$ quadrants $\mathcal{Q}=\mathcal{Q}_{R_1,\cdots, R_d}(t)$
(namely, for all the combinations that $R_1=``<"$,  or $``\ge" $, $\cdots$, $R_d=``<"$  or $ ``\ge"$), and (ii) $x(t)=x_{\mathcal{Q}_{\geq, \cdots, \geq}}$. In 
other words, $D(\RR^d)$ is the space of 
functions that are ``continuous from above with limits from below", which are similar   to  the space of 
the c\`{a}dl\`{a}g (French word abbreviation for 
``right continuous with left limits") functions
in one variable   (i.e., $d=1$). The metric   on $D(\RR^d)$ is introduced as follows. Let $\Lambda=\{\lambda:\RR^d\rightarrow \RR^d: \, \lambda(t_1,\cdots, t_d)=(\lambda_1(t_1),\cdots,\lambda_d(t_d))\}$, where each $\lambda_p: \RR \rightarrow \RR$ is continuous, strictly increasing and 
 has limits at both infinities.   Denote the  Skorohod  distance between $x,y \in D(\RR^d)$   by    
$$d(x,y)= \inf\{\min (\parallel x-y\lambda \parallel, \parallel \lambda \parallel):\, \lambda \in \Lambda\},$$
where $\displaystyle \parallel x-y\lambda \parallel= \sum_{n=1}^\infty 
\sup_{|t|\le n}    |x(t)-y(\lambda(t))| $ and $
\displaystyle \parallel \lambda \parallel=\sum_{n=1}^\infty 
\sup_{|t|\le n}   |\lambda(t)-t| $.  
 
\def\BB{\mathbb{B}}
Having introduced the metric  space 
$D(\RR^d)$  we can now explain  the concept of  weak convergence of a random measure
  on this space  with respect to its  
 Skorohod topology (the topology on $D(\RR^d)$ induced by the
Skorohod distance $d(x,y)$).
 Let $\mathbb{Q}_a: \Omega \times \mathcal{B}(\RR^d)\rightarrow [0, 1]$  be a family of  random probability measures 
  depending on a parameter $a>  0$ and  let $\BB: 
  \Omega \times \mathcal{B}(\RR^d)\rightarrow [0, 1]$  be another random probability measure. 
 Define 
 \[
 \mathbb{Q}_a(t_1, \cdots, t_d)=\mathbb{Q}_a((-\infty, t_1]\times \cdots \times
 (-\infty, t_d])\,,
 \quad (t_1, \cdots, t_d)\in \RR^d\,. 
 \]
\begin{definition}
We say $\mathbb{Q}_a$ converges to $\BB$ weakly on $D(\RR^d)$ with respect to the Skorohod topology,
denoted by   $\mathbb{Q}_a\stackrel{weakly}\rightarrow \BB$ in $D(\RR^d)$,  if for any bounded continuous 
(continuous 
with respect to Skorohod topology) functional  $f:D(\RR^d)\rightarrow \RR$
we have
\begin{equation}
\lim_{a\rightarrow \infty} \EE\!\! \left[f(\mathbb{Q}_a(\cdot, \cdots,\cdot))
\right]=\EE\!\! \left[f(\BB(\cdot, \cdots,\cdot))\right]\,. 
\end{equation}
  \end{definition}  
  
\subsection{Proof of Proposition \ref{prop.CRM}}

\begin{proof}

(i) For any $A\in\mathcal{X}$,  the $p$-th moment of $\mu_m(A)$ for any positive integer $p$ can be computed as follows. For any $\lambda>0$, we have
\begin{align}
\E\lk e^{-\lambda \mu_m(A)}\rk=\exp \ltd -H(A)  \int_0^{\infty}\lc 1-e^{-\lambda s} \rc \rho(ds)  \rtd\,.\label{eq: 5.1}
\end{align}
We shall denote $\phi(\lambda)=\int_0^{\infty}\lc 1-e^{-\lambda s} \rc \rho(ds)$.  Then, the $p$-th moment of $\mu_m(A)$ is
\begin{align}
\E\lk \lc\mu_m(A) \rc^p \rk&=(-1)^p \frac{\partial^p}{\partial \lambda^p} e^{-aH(A) {\phi}(\lambda)}\bigg|_{\lambda=0}\nonumber\\
&=(-1)^p \sum_{i=1}^p \lc -a H(A)\rc^i e^{-H(A) {\phi}(\lambda)} \sum_{\varrho : |\varrho|=i} \prod_{J \in \varrho} \frac{\partial^{|J|}}{\partial \lambda^{|J|}} {\phi}(\lambda)\bigg|_{\lambda=0}\,,\label{eq: 5.2}
\end{align}
where $\varrho$ represents partitions of $\{1,\cdots,p\}$ and the summation $\sum_{\varrho : |\varrho|=i} $ is taken of all partitions $\varrho$ such that $|\varrho|=i$. Equation \eqref{eq: 5.2} is a direct application of F{\` a\` a} di Bruno's formula, see e.g., \citep{hardy2006combinatorics}. By direct computation, we can rewrite \eqref{eq: 5.2} as follows.
\begin{align}
&\E\lk \lc\mu_m(A) \rc^p \rk\nonumber\\
&\qquad = (-1)^p \sum_{i=1}^p \lc - H(A)\rc^i e^{-H(A) {\phi}(\lambda)} \frac{1}{i!} \sum_{*}  {p \choose q_1, \cdots, q_i }\prod_{j=1}^i \frac{\partial^{q_j}}{\partial \lambda^{q_j}} {\phi}(\lambda)\bigg|_{\lambda=0}\,,\label{eq: 5.3}
\end{align}
where the summation $\sum_{*}$ is taken over all vectors $(q_1,\cdots, q_i)$ of positive integers such that $\sum_{j=1}^i q_j=p$. Note that 
\begin{align*}
&\frac{\partial^k}{\partial \lambda^k}{\phi}(\lambda)= (-1)^{k-1}\int_0^{\infty}s^k e^{-\lambda s} \rho(ds) \quad \text{for } k \geq 1 \,, \\
& \frac{\partial^k}{\partial \lambda^k}{\phi}(\lambda)\bigg|_{\lambda=0}=(-1)^{k-1} \tau_k\,.
\end{align*}
Thus,
\begin{align*}
\E\lk \lc\mu_m(A) \rc^p \rk&=(-1)^p \sum_{i=1}^p \lc - H(A)\rc^i \frac{1}{i!} \sum_{*}  {p \choose q_1, \cdots, q_i }\prod_{j=1}^i (-1)^{q_j-1} \tau_{q_j}\\
&=\sum_{i=1}^p  H(A)^i \frac{1}{i!} \sum_{*}  {p \choose q_1, \cdots, q_i }\prod_{j=1}^i  \tau_{q_j}\,.
\end{align*}

(ii) For any $A \in \mathcal{X}$, by the decomposition \eqref{e.2.8}, we have
 $\tilde{\mu}_a(A)=\sum_{m=1}^{\lfloor a \rfloor}\mu_m(A) + \mu_{\Delta}(A)$. It is sufficient to show $\mu_{\Delta}(A)=\int_{\lfloor a \rfloor}^a\int_0^{\infty}\int_A sN(d\zeta,ds,dx)$  is non-negative almost surely. This is true since $N(d\zeta,ds,dx)$ is a Poisson random measure with non-negative mean measure $d\zeta \rho(ds)H(dx)$. Thus, $\tilde{\mu}_a(A) \geq \tilde{\mu}_{\lfloor a \rfloor}(A)$ almost surely. Similarly,  $\tilde{\mu}_a(A) \leq \tilde{\mu}_{\lceil a \rceil}(A)$ almost surely. This proves result (ii).

(iii) Due to the fact that $\int_0^{\infty} \min (s,1)\rho(ds) <\infty$,  one has $\tau_k < \infty$ for any positive integer $k$ when $\tilde{\tau}_k<\infty$.

When $\tilde{\tau}_1<\infty$, we have $H(A)\tau_1 <\infty$.  Based on the decomposition $\tilde{\mu}_M=\sum_{m=1}^M \mu_m$, we have $\{\mu_m(A)\}_{m=1}^M$ is a sequence of iid random variables with finite mean $H(A)\tau_1$ for any given $A\in \mathcal{X}$. Result (iii) follows directly by applying the strong law of large numbers of a iid random sequence $\{\mu_m(A)\}_{m=1}^M$.

(iv) Also, by the decomposition $\tilde{\mu}_M=\sum_{m=1}^M \mu_m$, we can rewrite $R_M(A)$ as 
\begin{align}
R_M(A)=\frac{1}{\sqrt{M\tau_2H(A)(1-H(A))}} \sum_{m=1}^{M}\left( \mu_m(A)-H(A)\mu_m(\mathbb{X})\right)\,. 
\end{align}
Let $\{\mathbf{Y}_m\}_{m=1}^{M}$ be the sequence of $n$-dimensional random vector, where
\begin{align*}
\mathbf{Y}_m=\lc \frac{\mu_m(A_1)-H(A_1)\mu_m(\mathbb{X})}{\sqrt{\tau_2H(A_1)(1-H(A_1))}}, \cdots, \frac{\mu_m(A_n)-H(A_n)\mu_m(\mathbb{X})}{\sqrt{\tau_2H(A_n)(1-H(A_n))}} \rc^T
\end{align*}
Then, $\{\mathbf{Y}_m\}_{m=1}^{M}$ is an independent sequence with mean $0$ and covariance matrix $\Sigma=(\sigma_{ij})_{1\leq i,j \leq n}$ is given by 
\begin{align}
    \sigma_{ij}= \begin{cases}
    1 & \qquad \mbox{if $i=j$}\,,  \\
    -\sqrt{\frac{H(A_i)H(A_j)}{\left(1-H(A_i)\right)\left(1-H(A_j)\right)}}& \qquad \mbox{if $i \neq j$}\,.  \\
\end{cases}
\end{align}
The above covariance matrix is computed by using 
\begin{align*}
\E\lk \mu_k(A)^2 \rk=\frac{d^2}{d\lambda^2}\exp \ltd H(A) \int_0^{\infty} \lc 1-e^{-\lambda s} \rc \rho(ds)\rtd \bigg|_{\lambda=0}=H(A)^2\tau_1^2+H(A)\tau_2\,.
\end{align*}
By the relationship that $(R_M(A_1),\cdots, R_M(A_n))=\frac{\sum_{m=1}^M \mathbf{Y}_m}{\sqrt{M}}$, we can obtain
 result (iv) directly from the central limit theorem of independent $n$-dimensional random vectors $\{\mathbf{Y}_m\}_{m=1}^{M}$.
\end{proof}

\subsection{Proof of Proposition \ref{prop: order of var}}

\begin{proof}
The moment and variance of $P(A)$   are  given in Proposition 1 of \citep{james2006conjugacy}. The expression of $\mathcal{I}_a$ is given as follows.
\begin{align}
\mathcal{I}_a=a \int_0^{\infty} \lambda \exp\ltd -a\int_0^{\infty}\lc 1-e^{-\lambda s} \rc \rho(ds) \rtd \int_0^{\infty} s^2 e^{-\lambda s}\rho(ds) d\lambda\,.
\end{align}
Define a function $f(\lambda)$ as
\begin{align*}
f(\lambda):=-\int_0^{\infty}\lc 1-e^{-\lambda s} \rc \rho(ds)+\frac{\log \lambda}{a}+\frac{\log \lc\int_0^{\infty} s^2 e^{-\lambda s}\rho(ds) \rc}{a}\,.
\end{align*}
Then $\mathcal{I}_a=a \int_0^\infty e^{af(\lambda)}d\lambda$, by Laplace's method, the approximation of $\mathcal{I}_a$ is $a\sqrt{\frac{2\pi}{a|f''(\lambda_0)|}}e^{af(\lambda_0)}$ when $a \to \infty$.

First, $f'(\lambda)=-\int_0^{\infty} s e^{-\lambda s}\rho(ds)+\frac{1}{a\lambda}-\frac{\int_0^{\infty} s^3 e^{-\lambda s}\rho(ds)}{a\int_0^{\infty} s^2 e^{-\lambda s}\rho(ds)}=0$ implies $\lambda_0=\frac{1}{a\int_0^{\infty}s\rho(ds)}$ when $a$ is large. 

Then, $f''(\lambda)=\int_0^{\infty} s^2 e^{-\lambda s}\rho(ds)-\frac{1}{a\lambda^2}-\frac{\lc \int_0^{\infty} s^3 e^{-\lambda s}\rho(ds)\rc^2-\int_0^{\infty} s^4 e^{-\lambda s}\rho(ds)\int_0^{\infty} s^2 e^{-\lambda s}\rho(ds)}{a\lc\int_0^{\infty} s^2 e^{-\lambda s}\rho(ds) \rc^2}$. It is trivial to have $f''(\lambda_0)=-a+O\lc 1\rc$.

By the approximation that $e^{af(\lambda_0)}=O\lc \frac{1}{a}\rc$, we can conclude that $\mathcal{I}_a=O\lc \frac{1}{a}\rc$, which completes the proof.
\end{proof}

\subsection{Proof of Theorem \ref{thm: slln}}
\begin{proof}
Let $A \in \mathcal{X}$. By result (i) in Proposition \ref{prop.CRM}, one has
\begin{align}
\frac{\lfloor a \rfloor}{a}\frac{\tilde{\mu}_{\lfloor a \rfloor}(A)}{\lfloor a \rfloor}\leq \frac{\tilde{\mu}_a(A)}{a} \leq \frac{\lceil a \rceil}{a} \frac{\tilde{\mu}_{\lceil a \rceil}(A)}{\lceil a \rceil} \quad \text{a.s.}
\end{align}
When $a \rightarrow \infty$, we have $\lim_{a \to \infty}\frac{\lfloor a \rfloor}{a}=\lim_{a \to \infty}\frac{\lceil a \rceil}{a}=1$. Followed by result (ii) in Proposition \ref{prop.CRM}, we immediately obtain
\begin{align}
\lim_{a \to \infty} \frac{\lfloor a \rfloor}{a}\frac{\tilde{\mu}_{\lfloor a \rfloor}(A)}{\lfloor a \rfloor}=\lim_{a \to \infty} \frac{\lceil a \rceil}{a} \frac{\tilde{\mu}_{\lceil a \rceil}(A)}{\lceil a \rceil}=H(A)\tau_1 \quad \text{a.s.}
\end{align}
Therefore, $\lim_{a \to \infty}\frac{\tilde{\mu}_a(A)}{a}=H(A)\tau_1$ a.s.. Similarly, we have $\lim_{a \to \infty}\frac{\tilde{\mu}_a(\mathbb{X})}{a}=\tau_1$ a.s.. 

Thus, for any $A \in \mathcal{X}$, we have
\begin{align}
P(A)=\frac{\frac{\tilde{\mu}_a(A)}{a}}{\frac{\tilde{\mu}_a(\mathbb{X})}{a}} \overset{a.s.}{\rightarrow} \frac{H(A)\tau_1}{\tau_1}=H(A)\,. \label{eq:6.10}
\end{align}

To show $P \overset{a.s.}{\to} H$ as a random probability measure, we shall show $\lim_{a \to \infty} d(P,H)=0$ almost surely. Note that $f_i$ can be approximated by step functions $\sum_{k_i=1}^{M_i} \beta_{k_i=1}\mathbbm{1}_{A_{k_i}}$ with arbitrary error $\epsilon>0$ for any $f_i\in\{f_i\in C_b(\mathbb{X},\eta): i\geq 1\}$. Therefore, we have almost surely that
\begin{align*}
 \left|d(P,H)-\sum_{i=1}^{\infty} \frac{|\sum_{k_i=1}^{M_i} \beta_{k_i=1}\lc P(A_{k_i})-H(A_{k_i}) \rc| \wedge 1}{2^i} \right|<\epsilon\,.
\end{align*}
By \eqref{eq:6.10}, there exist a positive $a^*$ such that for all $a>a^*$, we almost surely have
\begin{align*}
d(P,H)&<\sum_{i=1}^{\infty} \frac{|\sum_{k_i=1}^{M_i} \beta_{k_i=1}\lc P(A_{k_i})-H(A_{k_i}) \rc| \wedge 1}{2^i} +\epsilon\\
&\leq \sum_{i=1}^{\infty} \frac{\sum_{k_i=1}^{M_i} |\beta_{k_i=1}| \left| P(A_{k_i})-H(A_{k_i}) \right| \wedge 1}{2^i} +\epsilon\\
&\leq \sum_{i=1}^{\infty} \frac{\epsilon}{2^i}+\epsilon=2\epsilon\,.
\end{align*}
Since $\epsilon>0$ is arbitrary, we have $\lim_{a \to \infty} d(P,H)=0$ almost surely. This proves result (i).

When $(\mathbb{X},\mathcal{X})=(\mathbb{R},\mathcal{B}(\mathbb{R}))$, the Glivenko-Cantelli theorem for $P$ follows directly by Theorem 20.6 in \citep{billingsley1995} as long as the strong law of large number result (i) holds.
\end{proof}
\subsection{Proof of Theorem \ref{thm: clt of NRMIs}}

\begin{proof}\label{proof: clt}
We shall prove (i) in Theorem \ref{thm: clt of NRMIs} first and prove (ii) afterwards based on the result  in (i).

We rewrite $D_a(\cdot)=\frac{R_a(\cdot)}{V_a}$, where 
\begin{align*}
&R_a(\cdot)=\frac{\tilde{\mu}_a(\cdot)-H(\cdot)\tilde{\mu}_a(\mathbb{X})}{\sqrt{a\tau_2H(\cdot)(1-H(\cdot))}}\,,\\
&V_a=\frac{\tilde{\mu}_a(\mathbb{X})}{a\tau_1}\,.
\end{align*}
For any disjoint measurable subsets $(A_1,\cdots,A_n)$, recall that $\bZ=(Z_1,\cdots,Z_n)$, we will prove the following two results when $a \rightarrow \infty$.  
\begin{align}
&\lc R_a(A_1),\cdots, R_a(A_n)\rc\overset{d}{\rightarrow} \bZ\,,\label{eq: 3.4}\\
&V_a \overset{p}{\rightarrow}1\,.\label{eq: 3.5}
\end{align}
Then, result (i) in Theorem \ref{thm: clt of NRMIs} follows directly by applying Slutsky’s theorem for \eqref{eq: 3.4} and \eqref{eq: 3.5}.

{\bf Step 1: Proof of \eqref{eq: 3.4}}

First, we can decompose $\tilde{\mu}_a(A)=\sum_{k=1}^{\lfloor a \rfloor}\mu_k(A)+\mu_{\Delta}(A)$ for any given $A \in \mathcal{X}$.

Therefore, we can rewrite $R_a$ for any  $A \in \mathcal{X}$ as
\begin{align}
R_a(A)=\frac{1}{\sqrt{a\tau_2H(A)(1-H(A))}} \sum_{k=1}^{\lfloor a \rfloor}\left( \mu_k(A)-H(A)\mu_k(\mathbb{X})\right) +\frac{\mu_{\Delta}(A)-H(A)\mu_{\Delta}(\mathbb{X})}{\sqrt{a\tau_2H(A)(1-H(A))}}
\end{align}
Let $\{\mathbf{Y}_k\}_{k=1}^{\lfloor a \rfloor}$ be the sequence of $n$-dimensional random vector, where
\begin{align*}
\mathbf{Y}_k=(\frac{\mu_k(A_1)-H(A_1)\mu_k(\mathbb{X})}{\sqrt{\tau_2H(A_1)(1-H(A_1))}}, \cdots, \frac{\mu_k(A_n)-H(A_n)\mu_k(\mathbb{X})}{\sqrt{\tau_2H(A_n)(1-H(A_n))}})^T
\end{align*}
Then, $\{\mathbf{Y}_k\}_{k=1}^{\lfloor a \rfloor}$ is an independent sequence with mean $0$ and covariance matrix $\Sigma$.

Similarly, we denote $\mathbf{Y}_{\Delta}=(\frac{\mu_{\Delta}(A_1)-H(A_1)\mu_{\Delta}(\mathbb{X})}{\sqrt{\tau_2H(A_1)(1-H(A_1))}}, \cdots, \frac{\mu_{\Delta}(A_n)-H(A_n)\mu_{\Delta}(\mathbb{X})}{\sqrt{\tau_2H(A_n)(1-H(A_n))}})^T$. It is trivial to show $\E \lk \mathbf{Y}_{\Delta} \rk=\mathbf{0}$, and thus,
$\frac{\mathbf{Y}_{\Delta}}{\sqrt{a}} \overset{p}{\to} \mathbf{0}$.

Thus,
\begin{align}
( R_a(A_1),\cdots, R_a(A_n))^T&=\frac{1}{\sqrt{a}}(\sum_{k=1}^{\lfloor a \rfloor} \mathbf{Y}_k+\mathbf{Y}_{\Delta})\nonumber\\
&=\sqrt{\frac{a}{\lfloor a \rfloor}} \sqrt{\lfloor a \rfloor}\left( \frac{\sum_{k=1}^{\lfloor a \rfloor} \mathbf{Y}_k}{\lfloor a \rfloor} \right)+\frac{\mathbf{Y}_{\Delta}}{\sqrt{a}}\,,
\end{align}
where $\sqrt{\lfloor a \rfloor}\left( \frac{\sum_{k=1}^{\lfloor a \rfloor} \mathbf{Y}_k}{\lfloor a \rfloor} \right) \overset{d}{\to} \bZ$ by result (iv) of Proposition \ref{prop.CRM} and 
$\frac{\mathbf{Y}_{\Delta}}{\sqrt{a}}\overset{p}{\to} \mathbf{0}$ as $a\to \infty$.

{\bf Step 2: Prove \eqref{eq: 3.5}.}

Followed by the Laplace transform in \eqref{eq: laplace}, we have the Laplace transform of $V_a$ as follows.
\begin{align}
\mathbb{E}\lk e^{-\lambda V_a} \rk&=\exp\ltd -a \int_0^{\infty} \lc 1-e^{-\frac{\lambda s}{a\tau_1}} \rc \rho(ds) \rtd\nonumber\\
&=\exp\ltd -a \int_0^{\infty} \lc \sum_{k=1}^{\infty}\frac{(-1)^{k-1}}{k!}\lc \frac{s\lambda}{a \tau_1}\rc^k \rc \rho(ds) \rtd\nonumber\\
&=\exp\ltd -\lambda + O\lc \frac{1}{a} \rc \rtd\,.\nonumber
\end{align}
With the assumption $\tilde{\tau}_2<\infty$, we have $\tau_1<\infty$. When $a \rightarrow \infty$, one has $V_a \overset{d}{\rightarrow} 1$ and thus $V_a \overset{p}{\rightarrow} 1$, which proves $\eqref{eq: 3.5}$.

As long as we have \eqref{eq: 3.4} and \eqref{eq: 3.5}, by Slutsky’s theorem we have 
\begin{align*}
(D_a(A_1),\cdots, D_a(A_n))=\frac{\bR_a}{V_a} \overset{d}{\rightarrow} \bZ\,,
\end{align*}
which implies result (i).

To prove the weakly convergence of $Q_{H,a}$ to $\mathbb{B}_H^o$ in $\mathbb{M}_{\mathbb{X}}$ with respect to the metric $d$, we shall prove for all bounded Lipschitz functions $g: \mathbb{M}_{\mathbb{X}} \rightarrow \mathbb{R}$,
\begin{align*}
d(\E[g(Q_{H,a})], \E[g(\mathbb{B}_H^o)])=\sum_{i=1}^{\infty} \frac{\left| \E\lk g\lc \int f_i dQ_{H,a} \rc \rk -\E\lk g\lc \int f_i d\mathbb{B}_H^o \rc \rk \right| \wedge 1}{2^i} \to 0
\end{align*}
when $a \to \infty$ and $f_i\in \{f_i\in C_b(\mathbb{X}, \eta): i \geq 1\}$.

From result (i), we have 
\begin{align*}
(Q_{H,a}(A_1), \cdots, Q_{H,a}(A_n)) \stackrel{d}{\rightarrow} (\mathbb{B}_H^o(A_1), \cdots, \mathbb{B}_H^o(A_n))\,.
\end{align*}
for any disjoint measurable subsets $A_1,\cdots,A_n$. Therefore,
\begin{align}
\left|\E \lk g \lc \int \sum_{j=1}^n \beta_j \mathbbm{1}_{A_j} dQ_{H,a} \rc \rk - \E \lk g\lc \int \sum_{j=1}^n \beta_j \mathbbm{1}_{A_j} d\mathbb{B}_H^o \rc \rk \right| \rightarrow 0 \,,\label{eq:6.15}
\end{align}
for all bounded Lipschitz functions $g$ when $a \rightarrow \infty$.

For any $f_i \in \{f_i \in C_b(\mathbb{X}, \eta): i\geq 1\}$, $f_i$ can be approximated by step functions $\sum_{k_i=1}^{M_i} \beta_{k_i=1}\mathbbm{1}_{A_{k_i}}$ with arbitrary error $\epsilon>0$. Therefore, 
\begin{align}
&\left| \E\lk g\lc \int f_i dQ_{H,a} \rc \rk -\E \lk g\lc \int \sum_{k_i=1}^{M_i} \beta_{k_i=1}\mathbbm{1}_{A_{k_i}}dQ_{H,a} \rc \rk \right|\nonumber\\
&\leq \E\lk \left| g\lc \int f_i dQ_{H,a} \rc - g\lc \int \sum_{k_i=1}^{M_i} \beta_{k_i=1}\mathbbm{1}_{A_{k_i}} dQ_{H,a} \rc \right| \rk\nonumber\\
&\leq K \E\lk \left|  \int f_i dQ_{H,a}  -  \int \sum_{k_i=1}^{M_i} \beta_{k_i=1}\mathbbm{1}_{A_{k_i}} dQ_{H,a} \right| \rk<\epsilon\,,\label{eq:6.16}
\end{align}
where the last inequality is due to $g$ is $K$-Lipschitz with some positive constant $K$.
Similarly, 
\begin{align}
\left| \E\lk g\lc \int f_i d\mathbb{B}_H^o \rc \rk -\E \lk g\lc \int \sum_{k_i=1}^{M_i} \beta_{k_i=1}\mathbbm{1}_{A_{k_i}}d\mathbb{B}_H^o  \rc \rk \right|<\epsilon\,.\label{eq:6.17}
\end{align}
By using \eqref{eq:6.15}, \eqref{eq:6.16} and \eqref{eq:6.17}, for all bounded Lipschitz functions $g$, when $a \rightarrow \infty$ we have
\begin{align*}
&\left| \E\lk g\lc \int f_i dQ_{H,a} \rc \rk -\E\lk g\lc \int f_i d\mathbb{B}_H^o \rc \rk \right|\\
&\leq \left| \E\lk g\lc \int f_i dQ_{H,a} \rc \rk -\E \lk g\lc \int \sum_{k_i=1}^{M_i} \beta_{k_i=1}\mathbbm{1}_{A_{k_i}}dQ_{H,a}  \rc \rk \right|+\\
&\quad\quad\left|\E \lk g \lc \int \sum_{k_i=1}^{M_i} \beta_{k_i=1}\mathbbm{1}_{A_{k_i}} dQ_{H,a} \rc \rk - \E \lk g\lc \int \sum_{k_i=1}^{M_i} \beta_{k_i=1}\mathbbm{1}_{A_{k_i}} d\mathbb{B}_H^o \rc \rk \right|+\\
&\quad\quad\left| \E\lk g\lc \int f_i d\mathbb{B}_H^o \rc \rk -\E \lk g\lc \int \sum_{k_i=1}^{M_i} \beta_{k_i=1}\mathbbm{1}_{A_{k_i}}d\mathbb{B}_H^o  \rc \rk \right|\\
&<3\epsilon\,.
\end{align*}
Therefore, when $a \rightarrow \infty$,
\begin{align*}
d(\E[g(Q_{H,a})], \E[g(\mathbb{B}_H^o)])=\sum_{i=1}^{\infty} \frac{\left| \E\lk g\lc \int f_i dQ_{H,a} \rc \rk -\E\lk g\lc \int f_i d\mathbb{B}_H^o \rc \rk \right| \wedge 1}{2^i}<\epsilon\,,
\end{align*}
for any $\epsilon>0$. This implies $Q_{H,a}$ converges weakly to $\mathbb{B}_H^o$ with respect to the metric $d$ in $\mathbb{M}_{\mathbb{X}}$.

When $\mathbb{X}=\mathbb{R}^d$, the second part of result (ii) of Theorem \ref{thm: clt of NRMIs} can be proved
  by obtaining the weak convergence of finite dimensional distributions 
  and by verifying  a tightness condition. 
  
  The finite dimensional weak convergence of $Q_{H,a}$ can be shown directly by part (i), i.e., for any $n \in \mathbb{Z}^+$ and any finite measurable sets $A_1, \cdots, A_n$ in $\mathcal{X}^d$, we have
\begin{align*}
(Q_{H,a}(A_1), \cdots, Q_{H,a}(A_n)) \stackrel{d}{\rightarrow} (\mathbb{B}_H^o(A_1), \cdots, \mathbb{B}_H^o(A_n))\,.
\end{align*}

For the tightness, by Theorem 2 of \citet{bickel1971},  it is sufficient to check the tightness condition, i.e., inequality (3) of \citet{bickel1971},  with  $\gamma_1=\gamma_2=2$, $\beta_1=\beta_2=1$ and $\xi=2H$.
Obviously,  $\xi$ is finite and nonatomic. For every pair of   Borel sets 
 $A$ and $B$ in $\mathcal{B}(\RR^d)$,  
  by Isserlis' theorem \citep{Isserlis_1918}, we   have  
\begin{align*}
&\mathbb{E}[|Q_{H,a}(A)|^2|Q_{H,a}(B)|^2]\\
&\qquad\qquad =[H(A)(1-H(A))][H(B)(1-H(B))]\mathbb{E}[D_a^2(A)D_a^2(B)]\\
&\qquad\qquad =[H(A)(1-H(A))][H(B)(1-H(B))]\left(1+2\frac{H(A)H(B)}{(1-H(A))(1-H(B))}+o(1)\right)\\
&\qquad\qquad  =3H(A)^2H(B)^2-H(A)^2H(B)-H(A)H(B)^2+H(A)H(B)+o(1)\\
&\qquad\qquad  \leq \xi(A)\xi(B).\ 
\end{align*}
The last inequality is due to the fact that $H(\cdot)\in (0,1)$ and thus $H(\cdot)^2 \leq H(\cdot)$.  Therefore, the tightness condition    on $D(\RR ^d)$
is verified.

\end{proof}

\subsection{Proof of Theorem \ref{thm: berry-esseen bound}}

We shall use the results in the following lemma in the proof of Theorem \ref{thm: berry-esseen bound}, we assume that $\tilde{\tau}_3<\infty$ in the following proofs.

\begin{lemma}\label{lemma: rate of V_a}
Recall that $V_a=\frac{\tilde{\mu}_a(\mathbb{X})}{a\tau_1}$. Assume that $\tilde{\tau}_3<\infty$. Let $\epsilon_a=\frac{1}{\sqrt{a}}$, $\delta_a=\frac{\tau_3}{\sqrt{a}\tau_1^3}$, then
\begin{align}
\mathbb{P}\lc |V_a-1|< \epsilon_a \rc \leq \delta_a\,.\label{eq: 5.11}
\end{align}
\end{lemma}
\begin{proof}
Based on the Laplace transform of $\tilde{\mu}$ in \eqref{eq: laplace}, for any $\lambda >0$, we can obtain the Laplace transform of $V_a$ as follows.
\begin{align}
\E\lk e^{-\lambda V_a} \rk= \exp\ltd -a\int_0^{\infty} \lc 1- e^{-\frac{\lambda s}{a \tau_1}} \rc \rho(ds) \rtd\,.
\end{align}
Thus, we can compute the moments of $V_a$ by $\E\lk V_a^m \rk = \lc -1\rc^m \frac{\partial^m}{\partial \lambda^m}\big|_{\lambda=0}$. We shall compute the first moment, the second moment, the third moment of $V_a$ as follows.

\begin{align*}
\E\lk V_a\rk&= -\frac{\partial}{\partial \lambda} \E\lk e^{-\lambda V_a} \rk \big|_{\lambda=0}\\
&= \exp\ltd -a\int_0^{\infty} \lc 1- e^{-\frac{\lambda s}{a \tau_1}} \rc \rho(ds) \rtd a \int_0^{\infty} e^{-\frac{\lambda s}{a \tau_1}} \frac{s}{a\tau_1} \rho(ds) \bigg|_{\lambda=0}\\
&=1\,.
\end{align*}
Similarly, we have
\begin{align*}
\E\lk V_a^2\rk&= \frac{\partial^2}{\partial \lambda^2} \E\lk e^{-\lambda V_a} \rk \big|_{\lambda=0}\\
&= \exp\ltd -a\int_0^{\infty} \lc 1- e^{-\frac{\lambda s}{a \tau_1}} \rc \rho(ds) \rtd \times\\
&\qquad \lk \lc a \int_0^{\infty} e^{-\frac{\lambda s}{a \tau_1}} \frac{s}{a\tau_1} \rho(ds) \rc^2+ \frac{1}{a\tau_1^2} \int_0^{\infty} e^{-\frac{\lambda s}{a \tau_1}} s^2 \rho(ds) \rk \bigg|_{\lambda=0}\\
&=1+\frac{\tau_2}{a\tau_1^2}\,.
\end{align*}
\begin{align*}
\E\lk V_a^3\rk&= -\frac{\partial^3}{\partial \lambda^3} \E\lk e^{-\lambda V_a} \rk \big|_{\lambda=0}\\
&= \exp\ltd -a\int_0^{\infty} \lc 1- e^{-\frac{\lambda s}{a \tau_1}} \rc \rho(ds) \rtd \bigg[
 \lc  \int_0^{\infty}\frac{1}{\tau_1} e^{-\frac{\lambda s}{a \tau_1}} s \rho(ds) \rc^3+\\
 &3 \lc  \frac{1}{a\tau_1^2} \int_0^{\infty} e^{-\frac{\lambda s}{a \tau_1}} s^2 \rho(ds) \rc \lc \int_0^{\infty}\frac{1}{\tau_1} e^{-\frac{\lambda s}{a \tau_1}} s \rho(ds)\rc + \frac{1}{a^2\tau_1^3} \int_0^{\infty} e^{-\frac{\lambda s}{a \tau_1}} s^3 \rho(ds) \bigg] \bigg|_{\lambda=0}\\
&=1+\frac{3\tau_2}{a\tau_1^2}+\frac{\tau_3}{a^2\tau_1^3}\,.
\end{align*}

It follows that $\E\lk \lc V_a-1\rc^3 \rk=\frac{\tau_3}{a^2\tau_1^3}$. By the Markov's inequality, we have 
\begin{align}
\mathbb{P}\lc |V_a-1| > \epsilon_a \rc \leq \frac{\E\lk( V_a-1)^3 \rk}{\epsilon_a^3}=a^{\frac{3}{2}}\frac{\tau_3}{a^2\tau_1^3}=\delta_a\,. \label{eq: 5.13}
\end{align}
\end{proof}

\begin{lemma}\label{lemma: difference of characters}
Recall that $R_a(\cdot)=\frac{\tilde{\mu}_a(\cdot)-H(\cdot)\tilde{\mu}_a(\mathbb{X})}{\sqrt{a\tau_2H(\cdot)(1-H(\cdot))}}$ and $\bR_a=(R_a(A_1), \cdots, R_a(A_n))$ for any pairwise disjoint measurable subsets $A_1,\cdots, A_n$. Let $\bZ$ be the n-dimensional normal random vector defined in Theorem \ref{thm: clt of NRMIs}. Assume that $\tilde{\tau}_3< \infty$. Then, for any $\bt>0$, when $a$ is large, we have
\begin{align}
\sup_{\bt \in \mathbb{R}^n}|F_{\bR_a}(\bt)-F_{\bZ}(\bt)| \leq \frac{C(n,H, \tau_3, \tau_2)}{ \sqrt{a}}\,,\label{eq: 5.14}
\end{align}
where $C(n,H, \tau_3, \tau_2)$ is a constant only depends on  $n$, $H$, $\tau_3$, $\tau_2$.
\end{lemma}
\begin{proof}

Recall Let $\{\mathbf{Y}_k\}_{k=1}^{\lfloor a \rfloor}$ be the sequence of $n$-dimensional random vector, where
\begin{align*}
\mathbf{Y}_k=(\frac{\mu_k(A_1)-H(A_1)\mu_k(\mathbb{X})}{\sqrt{\tau_2H(A_1)(1-H(A_1))}}, \cdots, \frac{\mu_k(A_n)-H(A_n)\mu_k(\mathbb{X})}{\sqrt{\tau_2H(A_n)(1-H(A_n))}})^T
\end{align*}
Then, $\{\mathbf{Y}_k\}_{k=1}^{\lfloor a \rfloor}$ is an independent sequence with mean $0$ and covariance matrix $\Sigma$.

Denote 
\[
\mathbf{Y}_{\Delta}=(\frac{\mu_{\Delta}(A_1)-H(A_1)\mu_{\Delta}(\mathbb{X})}{\sqrt{\tau_2H(A_1)(1-H(A_1))}}, \cdots, \frac{\mu_{\Delta}(A_n)-H(A_n)\mu_{\Delta}(\mathbb{X})}{\sqrt{\tau_2H(A_n)(1-H(A_n))}})^T\,.
\]  
It is trivial by Proposition \ref{prop.CRM} that  $(\{\mathbf{Y}_k\}_{k=1}^{\lfloor a \rfloor},  \mathbf{Y}_{\Delta})$   is a sequence of independent random vector with finite third moment when $\tilde{\tau}_3<\infty$. By our construction, $\bR_a=\frac{\sum_{k=1}^{\lfloor a \rfloor} \mathbf{Y}_k+\mathbf{Y}_{\Delta}}{\sqrt{\lceil a \rceil}} \times \frac{\sqrt{\lceil a \rceil}}{\sqrt{a}}$. Therefore, we can directly use Theorem 2 or Theorem 1 in \citet{kuchibhotla2020high} to obtain the result. In particular, the $\sigma_{\min}^2$ in \citet{kuchibhotla2020high} equal to $1$ in our case, and the and $\underline{\sigma}^2$ in \citet{kuchibhotla2020high} only depends on $(H(A_1), \cdots, H(A_n))$. 

\end{proof}

\begin{lemma}\label{lemma: BE of T_a}
Define random measure $T_a(\cdot):=\frac{\tilde{\mu}_a(\cdot)-H(\cdot)\tilde{\mu}_a(\mathbb{X})}{\sqrt{a\tau_2}}$ on $(\mathbb{X},\mathcal{X})$. Assume that $\tilde{\tau}_3< \infty$. Then, for any $t>0$ and any $f \in C_b(\mathbb{X},\eta)$, when $a$ is large, we have
\begin{align}
\sup_{t \in \mathbb{R}}|F_{\langle T_a, f \rangle}(t)-F_{\langle \mathbb{B}_H^o, f \rangle}(\bt)| \leq \frac{C(H, \tau_3, \tau_2)}{ \sqrt{a}}\,,\label{eq: 6.22}
\end{align}
where $C(H, \tau_3, \tau_2)$ is some constant only depends on  $H$, $\tau_2$, $\tau_3$.
\end{lemma}
\begin{proof}
Since $f$ can be approximated by step functions $\sum_{j=1}^M \beta_j \mathbbm{1}_{A_j}$ with arbitrary error $\epsilon>0$. Therefore, we choice the error $\epsilon < O(\frac{1}{a^2})$, then it is sufficient to show 
\begin{align}
\sup_{t \in \mathbb{R}} \left| \mathbb{P}\lc \sum_{j=1}^M \beta_j T_a(A_j) <t \rc - \mathbb{P}\lc \sum_{j=1}^M \beta_j \mathbb{B}_H^o(A_j) <t \rc \right|\leq \frac{C(H, \tau_3, \tau_2)}{ \sqrt{a}}\,.\label{eq: 6.23}
\end{align}
Let $W=\sum_{j=1}^M \beta_j T_a(A_j)$, then by the decomposition \eqref{e.2.8}, we have
\begin{align*}
W&=\sum_{j=1}^M \beta_j \frac{\tilde{\mu}_a(\cdot)-H(\cdot)\tilde{\mu}_a(\mathbb{X})}{\sqrt{a\tau_2}}\\
&=\sum_{j=1}^M \beta_j \frac{1}{\sqrt{a \tau_2}} \lk \sum_{k=1}^{\lfloor a \rfloor} \lc \mu_k(A_j) -H(A_j)\mu_k(\mathbb{X}) \rc + \lc \mu_{\Delta}(A_j)+H(A_j)\mu_{\Delta}(\mathbb{X}) \rc \rk\\
&=\sum_{k=0}^{\lfloor a \rfloor} w_k\,,
\end{align*}
where $w_k=\sum_{j=1}^M \beta_j \frac{1}{\sqrt{a \tau_2}} \lc \mu_k(A_j) -H(A_j)\mu_k(\mathbb{X}) \rc$ for $k \in \{1,\cdots, \lfloor a \rfloor\}$ and\\
 $w_0=\sum_{j=1}^M \beta_j \frac{1}{\sqrt{a \tau_2}}\lc \mu_{\Delta}(A_j)+H(A_j)\mu_{\Delta}(\mathbb{X}) \rc$.

Since $\{\mu_k\}_{k=1}^{\lfloor a \rfloor}$ and $\mu_{\Delta}$ are independent, $\{w_k\}_{k=0}^{\lfloor a \rfloor}$ is an independent sequence of random variables with mean $0$. The random variables $\{w_k\}_{k=0}^{\lfloor a \rfloor}$ have finite second and third moment if $\tilde{\tau}_3<\infty$. Therefore, the result \eqref{eq: 6.23} can be obtained by applying the Berry-Esseen theorem for independent random variables (see e.g. Theorem 3.6 in \citet{chen2010normal}).
\end{proof}
Now, we are ready to proof Theorem \ref{thm: berry-esseen bound}.

\begin{proof}\label{proof: berry-esseen thm}

(i) By the definition of $T_a$, $V_a$ and $Q_{H,a}$, we have $\langle Q_{H,a}, f \rangle=\frac{1}{V_a} \langle T_a, f\rangle$. For any $t > 0$,
\begin{align}
F_{\langle Q_{H,a}, f \rangle}(t) &\geq \mathbb{P}\lc \frac{\langle T_a, f\rangle}{V_a} \leq t, |V_a-1|<\epsilon_a\rc-\mathbb{P}\lc |V_a-1|>\epsilon_a \rc\nonumber\\
&\geq \mathbb{P} \lc \frac{\langle T_a, f\rangle}{V_a} \leq t,  |V_a-1|<\epsilon_a, V_a=1 \rc-\mathbb{P}\lc |V_a-1|>\epsilon_a \rc\nonumber\\
&=\mathbb{P}\lc \langle T_a, f\rangle \leq t \rc-\mathbb{P}\lc |V_a-1|>\epsilon_a \rc\,. \label{eq: 6.24}
\end{align}
On the other hand,
\begin{align}
F_{\langle Q_{H,a}, f \rangle}(t)=\mathbb{P}\lc \frac{\langle T_a, f\rangle}{V_a} \leq t\rc &\leq \mathbb{P}\lc \frac{\langle T_a, f\rangle}{V_a} \leq  t, |V_a-1|<\epsilon_a\rc+\mathbb{P}\lc |V_a-1|>\epsilon_a \rc\nonumber\\
&\leq \mathbb{P}\lc \langle T_a, f\rangle \leq (1+\epsilon_a) t \rc+\mathbb{P}\lc |V_a-1|>\epsilon_a \rc\,. \label{eq: 6.25}
\end{align}
We take $\epsilon_a=\frac{1}{\sqrt{a}}$, by Lemma \ref{lemma: rate of V_a}, $\mathbb{P}\lc |V_a-1|>\epsilon_a \rc \leq \frac{\tau_3}{\sqrt{a}\tau_1^3}$. Moreover, by Lemma \ref{lemma: BE of T_a}, we have 
\begin{align*}
&\mathbb{P}\lc \langle T_a, f\rangle \leq  t \rc \geq F_{\bZ}(\bt) -\frac{C(H, \tau_3, \tau_2)}{ \sqrt{a}}\,,\\
&\mathbb{P}\lc \langle T_a, f\rangle \leq (1+\epsilon_a) t \rc \leq F_{\bZ}((1+\epsilon_a) t) +\frac{C( H, \tau_3, \tau_2)}{ \sqrt{a}}=F_{\langle \mathbb{B}_H^o, f\rangle}( t)+O \lc \epsilon_a\rc +\frac{C(H, \tau_3, \tau_2)}{ \sqrt{a}}\,.
\end{align*}
Combine the above inequality with \eqref{eq: 6.24} and \eqref{eq: 6.25}, we have 
\begin{align}
|F_{\langle T_a, f\rangle}(t)-F_{\langle \mathbb{B}_H^o, f\rangle}( t)|\leq \frac{c_1}{\sqrt{a}}\,,
\end{align}
for any $t> 0$ and some constant $c_1$ that depends on $H$, $\tau_2$ and $\tau_3$. This completes the proof of result (i) in Theorem \ref{thm: berry-esseen bound}.

(ii) Recall that $D_a(\cdot)=\frac{R_a(\cdot)}{V_a}$, where 
\begin{align*}
&R_a(\cdot)=\frac{\tilde{\mu}_a(\cdot)-H(\cdot)\tilde{\mu}_a(\mathbb{X})}{\sqrt{a\tau_2H(\cdot)(1-H(\cdot))}}\,,\\
&V_a=\frac{\tilde{\mu}_a(\mathbb{X})}{a\tau_1}\,.
\end{align*}

Thus, $\mathbf{D}_a:=(D_a(A_1),\cdots, D_a(A_n))=\frac{\bR_a}{V_a}$. 

For any $\bt > \mathbf{0}$, we can have the lower and upper bound of the distribution function $F_{\mathbf{D}_a}(\bt)$ as follows.
\begin{align}
F_{\mathbf{D}_a}(\bt)=\mathbb{P}\lc \frac{\bR_a}{V_a} \leq \bt\rc &\geq \mathbb{P}\lc \frac{\bR_a}{V_a} \leq \bt, |V_a-1|<\epsilon_a\rc-\mathbb{P}\lc |V_a-1|>\epsilon_a \rc\nonumber\\
&\geq \mathbb{P}\lc \frac{\bR_a}{V_a} \leq \bt, |V_a-1|<\epsilon_a, V_a=1 \rc-\mathbb{P}\lc |V_a-1|>\epsilon_a \rc\nonumber\\
&=\mathbb{P}\lc \bR_a \leq \bt \rc-\mathbb{P}\lc |V_a-1|>\epsilon_a \rc\,. \label{eq: lower bound}
\end{align}
On the other hand, we have
\begin{align}
F_{\mathbf{D}_a}(\bt)=\mathbb{P}\lc \frac{\bR_a}{V_a} \leq\bt\rc &\leq \mathbb{P}\lc \frac{\bR_a}{V_a} \leq \bt, |V_a-1|<\epsilon_a\rc+\mathbb{P}\lc |V_a-1|>\epsilon_a \rc\nonumber\\
&\leq \mathbb{P}\lc \bR_a \leq (1+\epsilon_a)\bt \rc+\mathbb{P}\lc |V_a-1|>\epsilon_a \rc\,. \label{eq: upper bound}
\end{align}
We shall take $\epsilon_a=\frac{1}{\sqrt{a}}$, by Lemma \ref{lemma: rate of V_a}, $\mathbb{P}\lc |V_a-1|>\epsilon_a \rc \leq \frac{\tau_3}{\sqrt{a}\tau_1^3}$. On the other hand, by Lemma \ref{lemma: difference of characters}, we have 
\begin{align*}
&\mathbb{P}\lc \bR_a \leq \bt \rc \geq F_{\bZ}(\bt) -\frac{C(n,H, \tau_3, \tau_2)}{ \sqrt{a}}\,,\\
&\mathbb{P}\lc \bR_a \leq (1+\epsilon_a)\bt \rc \leq F_{\bZ}((1+\epsilon_a)\bt) +\frac{C(n,H, \tau_3, \tau_2)}{ \sqrt{a}}=F_{\bZ}(\bt)+O \lc \epsilon_a\rc +\frac{C(n,H, \tau_3, \tau_2)}{ \sqrt{a}}\,.
\end{align*}
Combine the above inequality with \eqref{eq: lower bound} and \eqref{eq: upper bound}, we have 
\begin{align}
|F_{\mathbf{D}_a}(\bt)-F_{\bZ}(\bt)|\leq \frac{c_2}{\sqrt{a}}\,,
\end{align}
for any $\bt> \mathbf{0}$ and some constant $c_2$ that depends on $n$, $H$, $\tau_2$ and $\tau_3$.
Thus, the results \eqref{eq: berry-esseen} follows immediately.

\end{proof}

\end{document}